\documentclass[11pt,leqno,amscd, amsfonts, amssymb,pstricks,verbatim]{amsart}
\usepackage{}
\usepackage{bbm}

\usepackage[all]{xy}

\usepackage{mathrsfs}
\usepackage{amsmath}
\usepackage{amssymb}
\usepackage{hyperref}

\usepackage{mathrsfs,color}

\newtheorem{thm}{Theorem}[section]

\newtheorem{cor}[thm]{Corollary}
\newtheorem{prop}[thm]{Proposition}

 \theoremstyle{definition}
\newtheorem{example}[thm]{Example}

\newtheorem{defn}[thm]{Definition}

\newtheorem{rem}[thm]{Remark}

\numberwithin{equation}{thm}

\def\ggg{\mathfrak{g}}

\def\nnn{\mathfrak{n}}
\def\hhh{\mathfrak{h}}
\def\bbb{\mathfrak{b}}

\def\bbz{\mathbb Z}
\def\bbc{\mathbb C}

\begin{document}

\title{The BGG Category for Generalised Reductive Lie Algebras}

\author{Ye Ren}


\begin{abstract}
A Lie algebra is said to be generalised reductive if it is a direct sum of a semisimple Lie algebra and a commutative radical. In this paper we extend the BGG category $\mathcal{O}$ over complex semisimple Lie algebras to the category $\mathcal {O'}$ over complex generalised reductive Lie algebras. Then we make a preliminary research on the highest weight modules and the projective modules in this new category $\mathcal {O'}$, and generalize some conclusions in the classical case. As a critical difference from the complex semisimple Lie algebra case, we prove that there is no projective module in $\mathcal {O'}$.
\end{abstract}

\maketitle

\section{Introduction}
In the Lie theory, BGG category $\mathcal{O}$ is a very important topic. It was first proposed by Joseph Bernstein, Israel Gelfand, Sergei Gelfand. (cf. \cite{BGG}). Nowadays, BGG category $\mathcal{O}$ has been applied in geometry, categorization, Kac-Moody algebra, quantum groups, physical mathematics and other aspects. In this paper, we say that a Lie algebra is generalised reductive if it is a direct sum of a semisimple Lie algebra and a commutative radical. The definition of generalized reductive Lie algebra is influenced by the definition of semi-reductive Lie algebra in \cite{SXY} and \cite{OSY}. To research the BGG category $\mathcal{O}$ over generalised reductive Lie algebras, we establish a new category $\mathcal{O'}$. In category $\mathcal{O'}$, we give the definition of Verma modules, irreducible modules, highest weight modules and projective modules. Beside that, we research the irreducibility of modules in $\mathcal{O'}$, and generalize the BGG strongly linked theorem in $\mathcal{O'}$. We also consider the relationship of standard filtration and BGG reciprocity between the BGG category and the category $\mathcal{O'}$. Particularly, we prove that there is no projective module in $\mathcal{O'}$.

\section{Preliminaries}

We first give some notations. Let $\ggg$ be a finite dimensional Lie algebra over complex field $\bbc$. And $\ggg=\ggg_0\bigoplus J$ where $\ggg_0$ is a semisimple Lie algebra and $J$ is the radical of $\ggg$, moreover $[J, J]=0$. We say $\ggg$ is a generalised reductive Lie algebra. It is easy to see that if $J$ is the center of $\ggg$, then $\ggg$ is a reductive Lie algebra. Let $\ggg_0=\nnn_0^-\oplus\hhh_0\oplus\nnn_0$ be the Cartan decomposition of $\ggg_0$. Let $\bbb_0=\hhh_0\oplus\nnn_0$, $\bbb=\bbb_0\bigoplus J$, $\nnn=\nnn_0\bigoplus J$. It is easy to see that $\bbb$ is a maximal solvable subalgebra of $\ggg$.

Denote $\Phi$ to be the root system of $\ggg_0$, and let $\Phi^+$ be the positive root system, $\Phi^-$is the negative root system. We choose a simple root system $\Delta=\{\alpha_1,...,\alpha_l \}\subset \Phi$. Weyl group $W$ is generated by all the reflections according to $\Phi$. The elements in $W$ are denoted as $s_\alpha$($\alpha\in \Phi$). Denote $\rho=\frac{1}{2}\sum\limits_{\alpha\in \Phi^+} {\alpha}$.

According to the Cartan decomposition of $\ggg_0=\hhh_0\oplus\bigoplus\limits_{\alpha\in\Phi}{\ggg_0}_{\alpha}$, we can choose a set of generators of $\ggg_0$ according to simple root system, denote as $e_i\in{\ggg_0}_{\alpha_i}$, $f_i\in{\ggg_0}_{-\alpha_i}$, $[e_i, f_i]=h_i\in\hhh_0$$(i=1,...,l)$. Denote $U(\ggg)$ as the universal algebra of $\ggg$. And $U(\ggg)_0$ is the universal algebra of $\ggg_0$.

Recall that a BGG category $\mathcal{O}$ is a subcategory of Mod$U(\ggg_0)$ whose objects satisfy following conditions (cf. \cite{Hum}):

($\mathcal{O}$1) $M$ is a finitely generated $U(\ggg_0)$-module.

($\mathcal{O}$2) $M$ is $\hhh_0$-semisimple, that is, $M$ is a weight module: $M=\bigoplus\limits_{\mu\in{\hhh_0^*}} M_{\mu}$, $M_{\mu}=\{v\in M\mid h\cdot v=\mu(h)v, \forall h\in \hhh_0\}$, $M_{\mu}$ is a weight space.

($\mathcal{O}$3) $M$ is locally $\nnn_0$-finite: for each $v\in M$, the subspace $U(\nnn)\cdot v$ of $M$ is finite dimensional.

($\mathcal{O}$4) All weight spaces of $M$ are finite dimensional.

\begin{defn}
Category $\mathcal{O'}$ is a subcategory of Mod$U(\ggg)$ whose  objects satisfy following conditions:

($\mathcal{O'}$1) $M$ is a finitely generated $U(\ggg)$-module.

($\mathcal{O'}$2) $M$ is $\hhh_0$-semisimple, that is, $M$ is a weight module: $M=\bigoplus\limits_{\mu\in{\hhh_0^*}} M_{\mu}$, $M_{\mu}=\{v\in M\mid h\cdot v=\mu(h)v, \forall h\in \hhh_0\}$, $M_{\mu}$ is a weight space.

($\mathcal{O'}$3) $M$ is locally $\nnn$-finite: for each $v\in M$, the subspace $U(\nnn)\cdot v$ of $M$ is finite dimensional.

($\mathcal{O'}$4) All weight spaces of $M$ are finite dimensional.
\end{defn}

\begin{prop}
Suppose $M\in\mathcal{O'}$, then $M$ is a finitely generated $U(\ggg_0)$-module.
\end{prop}
\begin{proof}
Suppose $U(\ggg)$-module $M$ has a set of generators $v_1,...,v_n$. Denote $V$ as a $U(\nnn)$-module generated by $v_1,...,v_n$. Then $V$ is finite dimensional. We can choose $w_1,..,w_m$ as a basis of $V$. It is obviously that $M$ is generated by $w_1,..,w_m$ as $U(\ggg_0)$-module.
\end{proof}

\begin{cor}
Suppose $M$ is an object of $\mathcal{O'}$ as $U(\ggg)$-module, then $M$ is an object of BGG category $\mathcal{O}$ as $U(\ggg_0)$-module.
\end{cor}

\begin{prop}
Suppose $M$ is an object of $\mathcal{O'}$, then $M$ is an artinian module.
\end{prop}
\begin{proof}
Suppose $M$ is not an artinian module , then there is an infinite $U(\ggg)$-module chain $$...\subseteq M_n\subseteq M_{n-1}\subseteq...\subseteq M_0=M.$$
It is also an infinite $U(\ggg_0)$-module chain, but $M$ is an artinian module in the BGG category $\mathcal{O}$, contradictory.
\end{proof}

\begin{prop}
$\mathcal{O'}$ is a noetherian category.
\end{prop}
\begin{proof}
Similar to the proof above.
\end{proof}

\section{Verma modules in category $\mathcal{O'}$}
We can view $J$ as a $\ggg_0$-module with adjoint action. Since $\ggg_0$ is a semisimple Lie algebra, then $J$ can be decomposed into a direct sum of a finite number of irreducible submodules
$$J=L(\lambda_1)\oplus L(\lambda_2)\oplus...\oplus L(\lambda_n).$$
we denote $J_1$ to be the direct sum of irreducible submodules of J with $\lambda_i= 0$ ($i\in \{1,2...,n\}$), and denote $J_2$ to be the direct sum of irreducible submodules of J with $\lambda_j\neq 0$ ($j\in \{1,2...,n\}$), then $J=J_1\oplus J_2$.

To define Verma module in category $\mathcal{O'}$, we first consider one-dimensional $U(\bbb)$-module $\bbc w_{(\lambda ,g)}$, $\lambda\in \hhh_0^*$, $g\in J^*$. $h\cdot w_{(\lambda ,g)}=\lambda(h)w_{(\lambda ,g)}$ for any $h\in \hhh_0$. $u\cdot w_{(\lambda ,g)}=g(u)w_{(\lambda ,g)}$ for any $u\in J$. The action of $\nnn_0$ on $w_{(\lambda ,g)}$ is follows by the proposition below.
\begin{prop}
For any $x\in \nnn_0$, $x\cdot w_{(\lambda ,g)}=0$.
\end{prop}
\begin{proof}
Without loss of generality, we may assume $x\in \nnn_0$ and $[h, x]=\alpha(h)x$ for any $h\in \hhh_0$, $\alpha \in \Phi^+$. If $x\cdot w_{(\lambda,g)}=cw_{(\lambda,g)},c\in \bbc$, then $h\cdot x\cdot w_{(\lambda,g)}=c \lambda(h)w_{(\lambda,g)}=[h,x]\cdot w_{(\lambda,g)}+x\cdot h\cdot w_{(\lambda,g)}=(c\alpha(h)+c \lambda(h))w_{(\lambda,g)}$. So $c\alpha(h)=0$ for any $h\in \hhh_0$. There is a $h_0\in \hhh_0$ such that $\alpha(h_0)\neq 0$. Thus $c$ must be zero. Finally we conclude $x\cdot w_{(\lambda,g)}=0$ for any $x\in \nnn_0$.
\end{proof}

To make one-dimensional space $\bbc w_{(\lambda ,g)}$ to be a $U(\bbb)$-module, there is no restriction to $\lambda\in\hhh_0^*$, but there are some restricted conditions about $g\in J^*$.
\begin{prop}
Suppose $x\in \nnn_0$, $h\in \hhh_0$, $u_1,u_2, u\in J$, then

(1)$g([h,u])=0.$

(2)$g([x,u])=0.$

(3)$g([u_1,u_2])=0.$
\end{prop}
\begin{proof}
(1)$h\cdot u\cdot w_{(\lambda,g)}=\lambda(h)g(u)w_{(\lambda,g)}$, $[h, u]\cdot w_{(\lambda,g)}+u\cdot h\cdot w_{(\lambda,g)}=g([h, u])w_{(\lambda,g)}+\lambda(h)g(u)w_{(\lambda,g)}$. Then we deduce that $g([h, u])=0$.

(2)$x\cdot u\cdot w_{(\lambda,g)}=0$. Since $[x, u]\in J$, $[x, u]\cdot w_{(\lambda,g)}+u\cdot x\cdot w_{(\lambda,g)}=g([x, u])w_{(\lambda,g)}$. We deduce that $g([e, u])=0$.

(3)$u_1\cdot u_2\cdot w_{(\lambda,g)}= g(u_1)g(u_2)w_{(\lambda,g)}$, $[u_1, u_2]\cdot w_{(\lambda,g)}+u_1\cdot u_2\cdot w_{(\lambda,g)}=g([u_1, u_2])w_{(\lambda,g)}+g(u_1)g(u_2)w_{(\lambda,g)}$. Then we deduce that $g([u_1, u_2])=0$.
\end{proof}

\begin{prop}
Suppose $u\in J_2$, then $g(u)=0$.
\end{prop}
\begin{proof}
Without loss of generality, we may assume $J_2=L(\lambda)$, $\lambda\neq 0$. We divided $L(\lambda)$ as a direct sum of wight spaces, $L(\lambda)=\bigoplus\limits_{\mu} L(\lambda)_{\mu}$. If $u\in L(\lambda)_{\mu}$, and $\mu \neq 0$, then there is a $h\in \hhh$ makes $\mu(h)\neq 0$ and $g([h, u])=\mu(h)g(u)=0$, it follows that g(u)=0. If $\mu= 0$, then $u\in L(\lambda)_0$. Suppose $u_0$ is a lowest weight vector of $L(\lambda)$, that is, $[\nnn_0^-,u_0]=0$. Then $u$ can be written as $[ \sum\limits_{r_1,...,r_l} k_{r_1,...,r_l} e_1^{r_1}e_2^{r_2}...e_l^{r_l}, u_0]$, $e_1,...,e_l\in\nnn_0, r_1,...,r_l\in\bbz^+, k_{r_1,...,r_l}\in\bbc$. Actually, $[e_1^{r_1}e_2^{r_2}...e_l^{r_l}, u_0]$ can be written as $[e, u']$ for some $e\in\nnn_0$ and for some $u'\in L(\lambda)_{\mu}$, $\mu \neq 0$. Since $g([e, u'])=0$, we get that $g(u)=0$. It follows that for any $u\in J_2$, $g(u)=0$.
\end{proof}

\begin{prop}
Suppose $u\in J_1$, $x\in \ggg_0$, then $[x,u]=0$.
\end{prop}
\begin{proof}
Without loss of generality, we may assume $J_1=L(0)$, dim$L(0)$=1. It is a 1 dimensional $\ggg_0$-module with adjoint action. If $x\in \hhh_0$, $[x,u]=x\cdot u=0$. If $x\in \nnn_0\oplus\nnn_0^-$, $h\cdot x=\alpha(h)x$ for any $h\in \hhh_0$, $\alpha \in \Phi$. Since dim $L(0)$=1, $x\cdot u=cu, c\in \bbc$. Then $h\cdot x\cdot u=0=[h,x]\cdot u+x\cdot h\cdot u=c\alpha(h)$. Thus for any $h\in \hhh_0$, $c\alpha(h)=0$. We can find a $h_0\in \hhh_0$ such that $\alpha(h_0)\neq 0$. It follows that $c=0$. Thus $[J_1,\ggg_0]=0$.
\end{proof}

We define a set $G=\{g\in J^*\mid g(J_2)=0\}.$

\begin{prop}
An one-dimensional $U(\bbb)$-module $\bbc_{(\lambda ,g)}$ is determined by a pair of functions $\lambda\in\hhh_0^*$ and $g\in G$. Conversely, suppose an one-dimensional $U(\bbb)$-module $\bbc_{(\lambda ,g)}$ is determined by a pair of functions $\lambda\in\hhh_0^*$ and $g\in J^*$, then $g\in G$.
\end{prop}
\begin{proof}
The second assertion is already proved, we now prove the first assertion. We already know $[J, J]=0, [\bbb_0, J_1]=0, [\bbb_0, J_2]\subset J_2, g\in G$. Suppose $x, y\in \bbb$, $x=a_1+b_1+c_1+d_1, y=a_2+b_2+c_2+d_2$, $a_1, a_2\in\hhh_0$, $b_1, b_2\in\bbb_0$, $c_1, c_2\in J_1$, $d_1, d_2\in J_2$, Ôò
\begin{eqnarray*}
x\cdot y\cdot w_{(\lambda ,g)}-y\cdot x\cdot w_{(\lambda ,g)}
&=&(a_1+b_1+c_1+d_1)\cdot (a_2+b_2+c_2+d_2)\cdot w_{(\lambda ,g)}\\
&&-(a_2+b_2+c_2+d_2)\cdot (a_1+b_1+c_1+d_1)\cdot w_{(\lambda ,g)}\\
&=&(\lambda(a_1)+g(c_1))(\lambda(a_2)+g(c_2))w_{(\lambda ,g)}\\
&&-(\lambda(a_2)+g(c_2))(\lambda(a_1)+g(c_1))w_{(\lambda ,g)}\\
&=&0.
\end{eqnarray*}
$$[x, y]\cdot w_{(\lambda ,g)}=[a_1+b_1+c_1+d_1, a_2+b_2+c_2+d_2]\cdot w_{(\lambda ,g)}=[c_1, c_2]\cdot w_{(\lambda ,g)}=0.$$
So $x\cdot y\cdot w_{(\lambda ,g)}-y\cdot x\cdot w_{(\lambda ,g)}=[x, y]\cdot w_{(\lambda ,g)}$. Then $\bbc w_{(\lambda ,g)}$ becomes a $U(\bbb)$-module.
\end{proof}

\begin{defn}
We define Verma module in category $\mathcal{O'}$£º
$$M(\lambda,g)=U(\ggg)\otimes_{U(\bbb)}{\bbc_{(\lambda ,g)}}.$$
$w_{(\lambda ,g)}$ is a nonzero element of one-dimensional $U(\bbb)$-module $\bbc_{(\lambda ,g)}$, and it satisfies: $h\cdot w_{(\lambda ,g)}= \lambda(h)w_{(\lambda ,g)}$ for any $h\in \hhh_0$. For any $u\in J$, $u\cdot w_{(\lambda ,g)}=g(u)w_{(\lambda ,g)}$. $\lambda \in \hhh_0^*, g \in G$.
\end{defn}

\section{The properties of category $\mathcal{O'}$}
\begin{defn}
Suppose $M\in\mathcal{O'}$, $v^+\in M$ and $\nnn_0\cdot v^+=0$. For any $h\in \hhh_0$, $h\cdot v^+=\lambda(h)v^+$, $\lambda \in \hhh_0^*$. For any $u\in J$, $u\cdot v^+=g(u)v^+$, $g \in G$. We say $v^+$ is a maximal vector of $M$ with highest weight $(\lambda, g)$. If $M$ is generated by $v^+$, we say $M$ is a highest weight module.
\end{defn}

\begin{rem}
The weights, weight modules and weight spaces mentioned in this paper are in the sense of semisimple Lie algebra. It is because they are based on the semisimple action of $\hhh_0$. The highest weights and highest weight modules mentioned in this paper are in the sense of generalised reductive Lie algebra, they are also highest weights and highest weight modules in the sense of semisimple Lie algebra.
\end{rem}

\begin{prop}
Suppose $M\in\mathcal{O'}$, and suppose there is a maximal vector in $M$ with wight $(\lambda, g)$. There is a $U(\ggg)$-module homomorphism from $M(\lambda ,g)$ to $M$.
\end{prop}

\begin{proof}
Suppose $w_{(\lambda ,g)}$ is the maximal vector of $M$ with weight $(\lambda, g)$. According to  the definition of Verma module, we can find a maximal vector $w$ of $M(\lambda,g)$, and the elements of $M(\lambda,g)$ can be expressed uniquely as $u\cdot w$, $u\in U(\nnn_0^-)$. We define a linear map
$$\Phi : M(\lambda,g)\rightarrow M.$$
$\Phi(u.w)=u\cdot w_{(\lambda ,g)}$, $u\in U(\nnn_0^-)$. We need to show $\Phi$ is a $U(\ggg)$-module homomorphism. That is to proof for any $x\in U(\ggg)$, $\Phi(xu\cdot w)=x\Phi(u\cdot w)$. According to PBW basis theorem, we write $xu=\sum\limits_i a_i b_i c_i u_i$, $a_i\in U(\nnn_0^-)$, $b_i\in U(\hhh_0)$, $c_i\in U(\nnn_0)$, $u_i\in U(J)$. $b_i c_i u_i\cdot w=\xi w$, $\xi\in \bbc$, $$\Phi(xu\cdot w)=\Phi(\sum\limits_i \xi a_i w)=\sum\limits_i \xi a_i u\cdot w_{(\lambda ,g)}.$$
$$x\Phi(u\cdot w)=xu\cdot w_{(\lambda ,g)}=\sum\limits_i \xi a_i u\cdot w_{(\lambda ,g)}.$$
Hence $\Phi(xu\cdot w)=x\Phi(u\cdot w)$.
\end{proof}

\begin{prop}
Suppose $M\in\mathcal{O'}$, then the submodules of $M$ are weight modules , that is $\hhh_0$-semisimple.
\end{prop}
\begin{proof}
Refer to the proof of theorem 10.9 in \cite{Car}. We know that $M=\bigoplus\limits_{\mu_i\in{\hhh_0^*}}M_{\mu_i}$, $i$ belongs to an index set $I$. Suppose $N$ is a submodule of $M$, $v\in N$, $v=\sum\limits_{\mu_i\in{\hhh_0^*}} v_{\mu_i}$, $v_{\mu_i}\in M_{\mu_i}$, there are only a finite number of sums that are not zero. We only need to prove every $v_{\mu_i}\in N$.
$$\prod\limits_{j\neq i}(h-\mu_j(h))v=\prod\limits_{j\neq i}(h-\mu_j(h))v_{\mu_i}=\prod\limits_{j\neq i}((\mu_i(h)-\mu_j(h))v_{\mu_i}, h\in \hhh_0.$$
We can find a $h\in \hhh$ such that $\mu_i(h)\neq \mu_j(h)$ for any $j\neq i$. For this $h\in \hhh$ we can get
$$\prod\limits_{j\neq i}((\mu_i(h)-\mu_j(h))v_{\mu_i}\in N.$$
Hence $v_{\mu_i}\in N$, and $N=\bigoplus\limits_{\mu_i\in{\hhh_0^*}}(M_{\mu_i}\bigcap N)$.
\end{proof}

\begin{prop}
$M(\lambda, g)$ has a unique maximal submodule.
\end{prop}
\begin{proof}

Each proper submodule of $M(\lambda, g)$ is a weight module. It cannot have $\lambda$ as a wight, since the one-dimensional space $M(\lambda, g)_\lambda$ generates $M(\lambda, g)$, and the sum of all proper submodule is still proper. Thus $M(\lambda, g)$ has a unique maximal submodule.

\end{proof}

\begin{defn}
Suppose $N$ is the maximal submodule of Verma module $M(\lambda, g)$, We define  $L(\lambda, g)$ as $M(\lambda, g)/N$.
\end{defn}
\begin{prop}
Suppose $M$ is an irreducible module in $\mathcal{O'}$, then $M$ is isomorphic to some $L(\lambda, g)$, $\lambda\in\hhh_0^*, g\in G$.
\end{prop}
\begin{proof}
Choose a nonzero vector $v^+$ of $M$. Then $v^+$ generates a finite dimensional $U(\nnn)$-module $V$. Since $\nnn$ is solvable, we can find a 1 dimensional submodule of V, choose a nonzero vector $w_{(\lambda,g)}$ $(\lambda \in \hhh_0^*, g \in G)$ of this submodule, $h.w_{(\lambda,g)}=\lambda(h)w_{(\lambda,g)}$ for any $h\in \hhh_0$, $u.w_{(\lambda,g)}=g(u)w_{(\lambda,g)}$ for any $u\in J$. $x.w_{(\lambda,g)}=0$ for any $x\in\nnn_0$. Since $M$ is irreducible, dim $V$=1 and $M$ is generated by $w_{(\lambda,g)}$. There is a surjective homomorphism from $M(\lambda, g)$ to $M$, so $M$ is isomorphic to $L(\lambda, g)$.

\end{proof}

\begin{prop}
Suppose $M$ is an object in $\mathcal{O'}$, then $M$ has a finite filtration with nonzero quotients each of which is a highest weight module.
\end{prop}

\begin{proof}
We may assume $M$ is generated by $v_1,...,v_n$. Let $V$ be a $U(\nnn)$-module generated by $v_1,...,v_n$. Since $M$ is locally $\nnn$-finite, the dimension of $V$ is finite. We use induction on dim$V$. If dim$V=1$, it is clear that $M$ itself is a highest weight module. Since $\nnn$ is solvable, we can find a 1 dimensional submodule of $V$, choose a nonzero vector $v^+$ of this submodule, it is obvious that $v^+$ is a maximal vector. It generates a $U(\ggg)$-submodule $M_1$, $M/M_1$ is still in the category $\mathcal{O'}$ and generated by $V/{\bbc v^+}$. Since dim$V/{\bbc v^+}<$ dim$V$, by induction $M/M_1$ has a finite filtration with nonzero quotients each of which is a highest weight module. The preimages of submodules of $M/M_1$ in the filtration with $M_1$ form a desired filtration of $M$.

\end{proof}

\begin{prop}
For any $v\in M(\lambda, g)$, $\lambda\in\hhh_0^*, g\in G$, and for any $u\in J$, $u\cdot v=g(u)v$.
\end{prop}
\begin{proof}
Suppose $w$ is a maximal vector of $M(\lambda, g)$ with highest weight $(\lambda, g)$, then every element of $M(\lambda, g)$ can be written as $v=x\cdot w$ for some $x\in U(\nnn_0^-)$. If $u\in J_1$, since $u$ commutes with $x$, $u\cdot x\cdot w=[u, x]\cdot w+x\cdot u\cdot w=g(u)x\cdot w$. If $u\in J_2$, since $[J_2, \ggg_0]\subset J_2$, $g\in G$ then $u\cdot x\cdot w=[u, x]\cdot w+x\cdot u\cdot w=0$. Thus for any $u\in J$ and any $v\in M(\lambda, g)$, $u\cdot v=g(u)v$.
\end{proof}
From the proposition above, we can see that $M(\lambda, g)\in\mathcal{O'}$ is $\hhh_0\bigoplus J$-semisimple.

\begin{cor}
Suppose $g_0\neq g_1$, dim$Hom_{\mathcal{O'}}(M(\mu, g_0),M(\lambda, g_1))=0$.
\end{cor}

\begin{cor}
Suppose $M(\mu, g')$ is a submodule of $M(\lambda, g)$, then $g'=g$.
\end{cor}

\begin{prop}
Suppose $M\in \mathcal{O'}$, then the action of $J_2$ on $M$ is nilpotent.
\end{prop}
\begin{proof}
Since $M$ is an object of $\mathcal{O'}$, there is a finite filtration of $M$
$$0=M_0\subset M_1\subset...\subset M_n=M.$$
$M_i$/$M_{i-1}$$\cong $L($\lambda_i, g_i$),$\lambda_i\in\hhh_0^*, g_i\in G$$(i=1, 2, ..., n)$. The action of $J_2$ on $L(\lambda_i, g_i)$ is zero. Then $J_2$ acts $n$ times on $M$ to be zero.
\end{proof}

\begin{prop}
Suppose $M\in \mathcal{O'}$, and $M$ equals $M(\lambda)$ as a $U(\ggg_0)$-module, then the action of $J_2$ is zero. It follows that $U(\ggg)$-module $M=M(\lambda, g)$ for some $g\in G$.
\end{prop}
\begin{proof}
Let $M$ decomposes to $M=\bigoplus\limits_{\mu\in\hhh_0^*}M_{\mu}$ according to $\hhh_0$. Choose $(0\neq)$ $v_{\lambda}\in M_{\lambda}$. dim$M_{\mu}\neq 0$ if and only if $(\lambda-\mu)\in\bbz^+\Phi$. $J_2$ decomposes to $J_2=\bigoplus\limits_{\gamma\in\hhh_0^*}J_{\gamma}$ with the adjoint action of $\hhh_0$, choose $(0\neq)u\in J_{\gamma}$. then $u\cdot v_{\lambda}\in M_{\lambda+\gamma}$. If $\gamma\notin\bbz^-\Phi$, then $u\cdot v_{\lambda}=0$. Denote $\Lambda=\{\gamma\in\hhh_0^*\mid \gamma\neq 0$, the action of $J_{\gamma}$ on $v_{\lambda}$ is not zero$\}$, obviously $\Lambda\subset\bbz^-\Phi$. We can define partial order on set $\Lambda$ $(\alpha\leq\beta$ if and only if $(\beta-\alpha)\in\bbz^+\Phi)$. $\Lambda$ is a finite set, if set $\Lambda$ is not empty, we can choose a minimal element $\tau$, let $(0\neq)$ $u\in J_{\tau}$. Since $u\cdot v_\lambda\neq0$,
there is a $x\in U(\nnn_0^-)$ ($x\notin\bbc$) such that $u\cdot v_{\lambda}=x\cdot v_{\lambda}$. $[u, x]\cdot v_{\lambda}=0=u\cdot x\cdot v_{\lambda}-x\cdot u\cdot v_{\lambda}$, so $u\cdot x\cdot v_{\lambda}=x\cdot x\cdot v_{\lambda}$, then the action of $u$ is not nilpotent, a contradiction. So $\Lambda$ is empty.  Thus the action of $J_{\gamma}$ $(\gamma\neq0)$ is zero. Let $(0\neq)$ $u\in J_0$, there is a $e\in\nnn_0$, $u'\in J_\tau$ $(\tau\neq 0)$ such that $[e, u']=u$, so $u\cdot v_{\lambda}=[e, u']\cdot v_{\lambda}=e\cdot u'\cdot v_{\lambda}-u'\cdot e\cdot v_{\lambda}=0$. Finitely we get the result that the action of $J_2$ on $v_{\lambda}$ is zero. Since $J_2$ is an ideal of $\ggg$, so the action of $J_2$ on $M(\lambda)$ is zero.
\end{proof}

\begin{prop}

(1)Suppose $U(\ggg)$-module $M$ is an object in $\mathcal{O'}$, and $M$ is irreducible as $U(\ggg_0)$-module, then $M$ is irreducible as $U(\ggg)$-module.

(2)Suppose $U(\ggg)$-module $N$ is a maximal submodule of $M(\lambda, g)$ as $U(\ggg_0)$-module, then $N$ is a maximal submodule of $M(\lambda, g)$ as $U(\ggg)$-module.

\end{prop}
\begin{proof}
(1)Suppose $0\subsetneqq N\subsetneqq M$ as $U(\ggg)$-modules, then $0\subsetneqq N\subsetneqq M$ as $U(\ggg_0)$-modules, this contradicts with the fact that $M$ is an irreducible $U(\ggg_0)$-module.

(2)Suppose $N\subsetneqq N'\subsetneqq M$ as $U(\ggg)$-modules, then$N\subsetneqq N'\subsetneqq M$ as $U(\ggg_0)$-modules, this contradicts with the fact that $N$ is a maximal $U(\ggg_0)$-submodule of $M(\lambda, g)$.
\end{proof}

\begin{prop}

(1)Suppose $N$ is a maximal submodule of $M(\lambda, g)$ as $U(\ggg)$-module, then $N$ is a maximal submodule of $M(\lambda)$ as $U(\ggg_0)$-module.

(2)Suppose $M$ is an irreducible $U(\ggg)$-module in category $\mathcal{O'}$, then $M$ is also an irreducible $U(\ggg_0)$-module.

\end{prop}
\begin{proof}
(1)Suppose $N\subsetneqq N'\subsetneqq M$ as $U(\ggg_0)$-modules. Let $N'$ be the set of generators, we get a $U(\ggg)$-module $N''$. Since for any $u\in J$ and $v\in N'$, $u\cdot v=g(u)v$. So $N'=N''$. Then $N\subsetneqq N'\subsetneqq M$ as $U(\ggg)$-modules, this contradicts with the fact that $N$ is a maximal $U(\ggg)$-submodule of $M(\lambda, g)$.

(2)Since $M$ is an irreducible $U(\ggg)$-module, we can find a maximal $U(\ggg)$-submodule $N$ of some $M(\lambda, g)$ such that $M$ is $U(\ggg)$-module isomorphic to $M(\lambda, g)/N$, $N$ is also a maximal $U(\ggg_0)$ submodule of $M(\lambda,g)$, so $M$ is $U(\ggg_0)$-module isomorphic to $M(\lambda, g)/N$. And $M$ is an irreducible $U(\ggg_0)$-module.

\end{proof}

\begin{prop}
Let $M(\lambda, g)$ be a Verma module in $\mathcal{O'}$ with maximal vector $w_{(\lambda, g)}$. Suppose $(\lambda+\rho)(h_i)\in \bbz^+\backslash\{0\}$, then ${f_i}^{(\lambda+\rho)(h_i)}w_{(\lambda, g)}$ generates a proper submodule of $M(\lambda, g)$. This submodule is isomorphic to $M(\mu, g)$ where $\mu+\rho=s_{\alpha_i}(\lambda+\rho)$.

\end{prop}
\begin{proof}
Suppose $n=(\lambda+\rho)(h_i)$, $uf_i^nw_{(\lambda, g)}=\sum\limits_{j=0}^{n} \binom{n}{j}(-1)^j f_i^{n-j} g((adf_i)^ju)w_{(\lambda, g)}$. If $u\in J_1$, then $(adf_i)^ju=0$, $j=1,2,...,n$. If $u\in J_2$, then $(adf_i)^ju\in J_2$ and $g((adf_i)^ju)=0$, $j=1,2,...,n$. So $uf_i^nw_{(\lambda, g)}=g(u)f_i^nw_{(\lambda, g)}$. It is easy to check that for any $x\in \nnn_0$, $xf_i^nw_{(\lambda, g)}=0$.
\end{proof}

In BGG category, there is a theorem(cf. \cite{Hum}): Let $\lambda, \mu \in \hhh_0^*$

(a)If $\mu$ is strongly linked to $\lambda$, then $U(\ggg_0)$-module $M(\mu)\hookrightarrow M(\lambda)$; In particular, $[M(\lambda): L(\mu)]\neq 0$.

(b)If $[M(\lambda): L(\mu)]\neq 0$, then $\mu$ is strongly linked to $\lambda$.
($M(\lambda)$ possesses a filtration with simple quotients isomorphic to various $L(\mu)$, $[M(\lambda): L(\mu)]$ is the multiplicity of $L(\mu)$.)

In category $\mathcal{O'}$, there is a similar theorem:

\begin{thm}

Let $\lambda, \mu \in \hhh_0^*, g\in G$

(a)If $\mu$ is strongly linked to $\lambda$, then $U(\ggg)$-module $M(\mu, g)\hookrightarrow M(\lambda, g)$; In particular, $[M(\lambda, g): L(\mu, g)]\neq 0$.

(b)If $[M(\lambda, g): L(\mu, g)]\neq 0$, then $\mu$ is strongly linked to $\lambda$.

\end{thm}
\begin{proof}

(a)If $\mu$ is strongly linked to $\lambda$, then there is a $U(\ggg_0)$-module homomorphism $\phi: M(\mu, g)\hookrightarrow M(\lambda, g)$. Let $x\in M(\mu, g)$, $u\in J$. Then $\phi(u\cdot x)=\phi(g(u)x)=g(u)\phi(x)=u\cdot\phi(x)$. So $\phi$ can be extended to be a $U(\ggg)$-module homomorphism.

(b)$[M(\lambda, g): L(\mu, g)]\neq 0$, then$[M(\lambda, g): L(\mu, g)]\neq 0$ as $U(\ggg_0)$-modules, so $\mu$ is strongly linked to $\lambda$.

\end{proof}

\section{Projective modules in category $\mathcal{O'}$}
\begin{defn}
Let $P\in\mathcal{O'}$, for any $U(\ggg)$-module homomorphism $\phi:P\rightarrow N$, $N\in\mathcal{O'}$, and for any $U(\ggg)$-module epimorphism $\pi: M\rightarrow N$, $M\in\mathcal{O'}$, there is a $U(\ggg)$-module homomorphism $\psi:P\rightarrow M$, such that $\pi\circ\psi=\phi$, then $P$ is a projective module in category $\mathcal{O'}$.
\end{defn}

\begin{rem}
Let $M\in\mathcal{O'}$, if $M$ is a projective $U(\ggg_0)$-module in BGG category $\mathcal{O'}$, then $M$ may not be a projective $U(\ggg)$-module. Indeed, There is no projective module in category $\mathcal{O'}$.
\end{rem}

\begin{example}
We first give a example that is projective as $U(\ggg_0)$-module but not projective as $U(\ggg)$-module.
Let $\ggg_0=sl_n(\bbc)$, $\ggg=gl_n(\bbc)=\ggg_0\oplus \bbc z$.
If $\lambda$ is dominant, $g(z)=3$, then $M(\lambda, g)$ is a projective $U(\ggg_0)$-module in BGG category(cf.\cite{Hum} 3.8). We will show $M(\lambda, g)$ is not projective $U(\ggg)$-module in category $\mathcal{O'}$.

Suppose $w$ is a maximal vector of $M(\lambda, g)$ with weight $(\lambda, g)$. Denote $U(\ggg_0)$-module $L_1\cong L_2\cong L(\lambda)$, let $v_1$ be a maximal vector of $L_1$ with weight $\lambda$, and $v_2$ is a maximal vector of $L_2$ with weight $\lambda$. Let $z\cdot v_1=3v_1+v_2$ and $z\cdot v_2=3v_2$, then $L_1\oplus L_2$ becomes a $U(\ggg)$-module. Let $\varphi$ be a $U(\ggg)$-module homomorphism from $M(\lambda, g)$ to $L(\lambda, g)$, $\varphi$ can also be viewed as a $U(\ggg_0)$-module homomorphism. Let $v^+=\varphi(w)$ be a maximal vector of $L(\lambda, g)$ with weight $(\lambda, g)$.

Suppose $\pi$ is a $U(\ggg_0)$-module epimorphism: $k_1v_1+k_2v_2 \rightarrow k_1v^+$, $k_1, k_2\in \bbc$. $\bar{\varphi}(w)$ is a weight vector of $L_1\oplus L_2$ with weight $\lambda$.
Since $M(\lambda, g)$ is a projective $U(\ggg_0)$-module, there is a $U(\ggg_0)$-module commutative diagram:
\begin{equation*}
  \begin{array}{c}
\xymatrix{
  & M(\lambda, g)\ar[dl]_{\bar{\varphi}}\ar[d]^{\varphi} \\
  L_1\oplus L_2\ar[r]^{\pi} & L(\lambda, g)\ar[r]^{}& 0.}
\end{array}
\end{equation*}
Since $\pi(z\cdot v_1)=\pi(3v_1+v_2)=3v^+=z\cdot \pi(v_1)$, $\pi(z\cdot v_2)=\pi(3v_2)=0=z\cdot \pi(v_2)$, and $z$ commutes with $\ggg_0$, then $\pi$ is also a $U(\ggg)$-module homomorphism. We will show that the diagram above is not a $U(\ggg)$-module commutative diagram.
If it is a $U(\ggg)$-module commutative diagram, assume $\pi(v_1+f_2v_2)=v^+$, $f_2\in U(\nnn_0^-)$, and $\bar{\varphi}(w)=(v_1+f_2v_2)$. But the map $\bar{\varphi}$ can not be a $U(\ggg)$-module homomorphism. It is because
$$z\cdot \bar{\varphi}(w)=z\cdot (v_1+f_2v_2)=3v_1+v_2+3f_2v_2.$$
$$\bar{\varphi}(z\cdot w)=3\bar{\varphi}(w)=3v_1+3f_2v_2.$$
$$z\cdot \bar{\varphi}(w)\neq\bar{\varphi}(z\cdot w).$$
Thus $M(\lambda, g)$ is not a projective $U(\ggg)$-module.
\end{example}

\begin{prop}
There is no projective module in category $\mathcal{O'}$. In other words, the projective modules in $U (\ggg)$-module category are not in category $\mathcal{O'}$.
\end{prop}
\begin{proof}
Let $J$ be a $\ggg_0$-module with adjoint action, then $J$ can be decomposed to $J=L(\lambda_1)\oplus L(\lambda_2)\oplus...\oplus L(\lambda_n)$. Since $L(\lambda_1)$ is finite dimensional, we can find a $u\in L(\lambda_1)$ satisfies $h\cdot u=\alpha(h)u$ for any $h\in\hhh_0$ and $\nnn_0^-\cdot u=[\nnn_0^-, u]=0$.

Suppose $P$ is a projective module in category $\mathcal{O'}$, let $N$ be a maximal submodule of $P$, and $P/N\cong L(\gamma, g)$. Then there is a nature $U(\ggg)$-module homomorphism $\phi_1$ from $P$ to $P/N\cong L(\gamma, g)$. Let $v$ be a maximal vector of $L(\gamma, g)$, $w_1\in P$, $\phi_1(w_1)=v$.

We now construct a special $U(\ggg)$-module. Let $L(\gamma)$ and $L(\gamma+\alpha)$ be two irreducible $U(\ggg_0)$-modules. Let $v_1$ be a maximal vector of $L(\gamma)$ and $v_2$ be a maximal vector of $L(\gamma+\alpha)$. We define $u\cdot v_1=g(u)v_1+v_2$ and $u\cdot v_2=g(u)v_2$. Each element of $L(\gamma)$ can be written as $f_1v_1$ for some $f_1\in U(\nnn_0^-)$. Since $u$ commutes with $\nnn_0^-$, so $u\cdot f_1v_1=g(u)f_1v_1+f_1v_2$. Each element of $L(\gamma+\alpha)$ can be written as $f_2v_2$ for some $f_2\in U(\nnn_0^-)$. So $u\cdot f_2v_2=g(u)f_2v_2$. Thus we have defined the action of $u$ on $L(\gamma)\oplus L(\gamma+\alpha)$. Since $u$ is a generator of $L(\lambda_1)$, then we can define the action of $L(\lambda_1)$ on $L(\gamma)\oplus L(\gamma+\alpha)$. Let the action of $L(\lambda_2)\oplus...\oplus L(\lambda_n)$ to be a scalar according to function $g$, then $L(\gamma)\oplus L(\gamma+\alpha)$ becomes a $U(\ggg)$-module.

Next, we define a $U(\ggg)$-module epimorphism from $L(\gamma)\oplus L(\gamma+\alpha)$ to $L(\gamma, g)$. Let $\pi_1(f_1v_1+f_2v_2)=f_1v$, $f_1,f_2\in U(\nnn_0^-)$. It is easy to check that $\pi_1$ is a $U(\ggg)$-module epimorphism.

Since $P$ is a projective module in category $\mathcal{O'}$, there is a $U(\ggg)$-module homomorphism $\phi_2$ satisfies $\pi_1\circ\phi_2=\phi_1$. Then we have a commutative diagram
\begin{equation*}
  \begin{array}{c}
\xymatrix{
  & P\ar[dl]_{\phi_2}\ar[d]^{\phi_1} \\
  L(\gamma)\oplus L(\gamma+\alpha)\ar[r]^-{\pi_1} & L(\gamma, g)\ar[r]^{}& 0.}
\end{array}
\end{equation*}
Since $\pi_1\circ\phi_2(w_1)=\phi_1(w_1)=v$, $\phi_2(w_1)\in\pi_1^{-1}(v)=v_1+L(\gamma+\alpha)$. Suppose $\phi_2(w_1)=v_1+v'$, $v'\in L(\gamma+\alpha)$.
Let $w_2=u\cdot w_1-g(u)w_1$, we show $w_2\neq 0$. If $u\cdot w_1=g(u)w_1$, then
$$u\cdot\phi_2(w_1)=u\cdot (v_1+v')=g(u)v_1+v_2+g(u)v'$$
$$\phi_2(u\cdot w_1)=\phi_2(g(u)w_1)=g(u)v_1+g(u)v'$$
$$u\cdot\phi_2(w_1)\neq\phi_2(u\cdot w_1)$$
So $w_2\neq 0$. Next we show $w_1, w_2$ are linearly independent. Suppose $k_1w_1+k_2w_2=0$, then $\phi_1(k_1w_1+k_2w_2)=k_1\phi_1(w_1)=k_1v=0$, so $k_1=0$. Since $w_2\neq 0$, we have $k_2=0$.

Again, we construct a special $U(\ggg)$-module. Consider $U(\ggg_0)$-module $L(\gamma)\oplus L(\gamma+\alpha)\oplus L(\gamma+\alpha)$. Let $v_1, v_2, v_3$ be  maximal vectors of three irreducible $U(\ggg_0)$-module respectively. We define $u\cdot v_1=g(u)v_1+v_2+v_3$, $u\cdot v_2=g(u)v_2+v_3$ and $u\cdot v_3=g(u)v_3$. Let the action of $L(\lambda_2)\oplus...\oplus L(\lambda_n)$ to be scalars according to function $g$. Then $L(\gamma)\oplus L(\gamma+\alpha)\oplus L(\gamma+\alpha)$ becomes a $U(\ggg)$-module similarly. Define a $U(\ggg)$-module homomorphism $\pi_2$ to be a projection from $L(\gamma)\oplus L(\gamma+\alpha)\oplus L(\gamma+\alpha)$ to $L(\gamma)\oplus L(\gamma+\alpha)$.

If $P$ is a projective module in category $\mathcal{O'}$, there is a commutative diagram

\xymatrix{
 & & P\ar[dll]_-{\phi_3}\ar[dl]_{\phi_2}\ar[d]^{\phi_1} \\
 L(\gamma)\oplus L(\gamma+\alpha)\oplus L(\gamma+\alpha)\ar[r]^-{\pi_2}& L(\gamma)\oplus L(\gamma+\alpha)\ar[r]^-{\pi_1} & L(\gamma, g)\ar[r]^{}& 0.}

Let $w_3=u\cdot w_2-g(u)w_2$, we can prove that  $w_3\neq 0$ and $w_1, w_2, w_3$ are linearly independent. Proceeding above procedures, we can find linearly independent elements $w_1, w_2,..., w_n$ in $U(\nnn)\cdot w_1$, $n$ is sufficiently large. It contradicts with the fact that $P$ is locally $\nnn$-finite. So $P$ is not a projective module in category $\mathcal{O'}$.

\end{proof}

\begin{prop}
Let $M\in\mathcal{O'}$, if $M$ has a $U(\ggg_0)$-module standard filtration, then $M$ has a $U(\ggg)$-module standard filtration. Their length are the same.
\end{prop}
\begin{proof}
Since $M$ has a $U(\ggg_0)$-module standard filtration, then $M$ is $U(\nnn_0^-)$-free (cf. \cite{Hum} 3.7). Suppose $v_1,...,v_n$ generate $U(\ggg)$-module $M$, let $V$ be a $U(\nnn)$-module generated by $v_1,...,v_n$, then $V$ is finite dimensional. Since $\nnn$ is solvable, we can find a suitable basis of $V$ such that the action matrix of $\nnn$ is a upper triangular matrix, that is, there is a basis $w_1, w_2,..., w_k$ of $V$ such that $U(\ggg)$-modules $M_i(i=1,...,k)$ are generated by $w_i,...,w_k$. And there is a $U(\ggg)$-module standard filtration
$$0=M_0\subset M_1\subset...\subset M_k=M.$$
$M_0=M(\lambda_0, g_0)$, $M_i/M_{i-1}\cong M(\lambda_i, g_i)(\lambda_i\in\hhh_0^*, g_i\in G, i=1, 2, ..., k)$.
It is also a $U(\ggg_0)$-module standard filtration. The uniqueness of the length of standard filtration shows their length are the same.
\end{proof}

In the BGG category, $P(\lambda)$ is a projective cover of $L(\lambda)$. There is a theorem(BGG Reciprocity)(cf. \cite{Hum}): Let $\lambda, \mu\in \hhh_0^*$. Denote the multiplicity with which each Verma module $M(\mu)$ occurs in a standard filtration of $P(\lambda)$ by $(P(\lambda): M(\mu))$. Then $(P(\lambda): M(\mu))=[M(\mu): L(\lambda)].$

According to the properties of projective cover, $P(\lambda)$ is generated by one element. Assume $P(\lambda)$ is generated by $v_{\lambda}$, for any $u\in J$, let $u\cdot v_{\lambda}=g(u)v_{\lambda}$, $g\in J^*$ (indeed, $g\in G$), then $U(\ggg_0)$-module $P(\lambda)$ become a $U(\ggg)$-module, denoted as $P(\lambda, g)$. According to the proposition above, we know that $P(\lambda, g)$ has a $U(\ggg)$-module standard filtration. And there is a theorem in category $\mathcal{O'}$:
\begin{thm}
Let $\lambda, \mu\in \hhh_0^*$. then
$$(P(\lambda, g): M(\mu, g))=[M(\mu, g): L(\lambda, g)].$$
\end{thm}
\begin{proof}
There is a $U(\ggg)$-module standard filtration of $P(\lambda, g)$:
$$0=M_0\subset M_1\subset...\subset M_n=M.$$ $M_n=P(\lambda, g)$, $M_i$/$M_{i-1}$$\cong $M($\lambda_i, g_i$)(i=1, 2, ..., n). We first show $g_i=g$. Since $M_n/M_{n-1}\cong M(\lambda_n, g_n$), there is a $U(\ggg)$-module nature homomorphism:
$$\phi : P(\lambda, g)\rightarrow P(\lambda, g)/M_{n-1}= M(\lambda_n, g_n).$$
$$\phi(v_{\lambda})=x\cdot w_{(\lambda_n, g_n)}.$$
$v_{\lambda}$ generates $P(\lambda, g)$, $w_{(\lambda_n, g_n)}$ is a maximal vector of $M(\lambda_n, g_n)$, $x\in U(\nnn_0^-)$. For any $u\in J$, $u\cdot \phi(v_{\lambda})=\phi(u\cdot v_{\lambda})$. On the one hand, $u\cdot \phi(v_{\lambda})=u\cdot x\cdot w_{(\lambda_n, g_n)}=g_n(u)x\cdot w_{(\lambda_n, g_n)}$. On the other hand, $\phi(u\cdot v_{\lambda})=\phi(g(u)v_{\lambda})=g(u)x\cdot w_{(\lambda_n, g_n)}$. So $g_n=g$ and $g\in G$. Suppose $y\cdot v_{\lambda}$, $y\in U(\ggg_0)$ is an element of  $P(\lambda, g)$, then for any $u\in J$, $u\cdot y\cdot v_{\lambda}=g(u)y\cdot v_{\lambda}$. Hence $g_i=g$, $i=1, ..., n$. From this result we know that different selected generators in $P(\lambda)$ generate the same $P(\lambda, g)$. Moreover,
$$(P(\lambda, g): M(\mu, g))=(P(\lambda): M(\mu))=[M(\mu): L(\lambda)]=[M(\mu, g): L(\lambda, g)].$$
\end{proof}

\subsection*{Acknowledgements}
I would like to thank professors Bin Shu, Lei Lin, Jun Hu for their helpful comments and suggestions.

\end{document}